\documentclass[11pt,a4paper,reqno]{amsart}

\usepackage{amsmath}
\usepackage{amsfonts}
\usepackage{a4wide}
\usepackage{amssymb}
\usepackage{amsthm}

\theoremstyle{plain}

\newtheorem{teo}{Theorem}[section]

\newtheorem{prop}[teo]{Proposition}

\newtheorem{ackn}{Acknowledgments\!}

\theoremstyle{definition}

\theoremstyle{remark}

\numberwithin{equation}{section}

\def\eps{\varepsilon}

\def\f{\varphi}

\title[Critical metrics of the $L^{2}$--norm of the scalar curvature]{Critical metrics of the $L^{2}$--norm \\of the scalar curvature}
\date{\today}

\author[Giovanni Catino]{Giovanni Catino}
\address[Giovanni Catino]{Dipartimento di Matematica, Politecnico di Milano, Piazza Leonardo da Vinci 32, 20133 Milano, Italy}
\email[]{giovanni.catino@polimi.it}

\date{\today}

\begin{document}

\begin{abstract} In this paper we investigate complete critical metrics of the $L^{2}$--norm of the scalar curvature. We prove that any complete critical metric with positive scalar curvature has constant scalar curvature and we characterize critical metrics with nonnegative scalar curvature in dimension three and four.
\end{abstract}

\maketitle

\begin{center}


\medskip

\centerline{\bf AMS subject classification:  53C24, 53C25}

\end{center}

\

\section{Introduction}

%

Let $(M^{n},g)$, $n\geq 3$, be a $n$--dimensional smooth Riemannian manifold and consider the functional
\begin{equation}\label{action}
\mathcal{S}^{2}(g) \,=\, \int_{M} R^{2}\,dV_{g} \,
\end{equation}
on the space of Riemannian metrics on $M^{n}$, where $R$ and $dV_{g}$ denote the scalar curvature and the volume form of $g$ respectively. The first variation of $\mathcal{S}^{2}$ (see~\cite{besse}) in the direction of $h$ reads
\begin{eqnarray*}
\delta\,\mathcal{S}^{2} (g)[h] &=& \int_{M} \big( 2R\,\delta R + \tfrac{1}{2}R^{2} \,tr(h) \big) \,dV_{g} \\
&=& \int_{M} \big( -2R \,\Delta tr(h) + 2R\,\nabla_{i}\nabla_{j} h_{ij}-2R\, R_{ij}h_{ij} +\tfrac{1}{2}R^{2} \,tr(h) \big) \,dV_{g} \\
&=& \int_{M} \big( -2\Delta R \,g_{ij}+2\nabla_{i}\nabla_{j}R -2R R_{ij} + \tfrac{1}{2}R^{2} \,g_{ij} \big) h_{ij} \,dV_{g}\,.
\end{eqnarray*}
Hence, the Euler--Lagrange equation for a critical metric of $\mathcal{S}^{2}$ is given by
$$
2R \, Ric - 2\nabla^{2} R + 2\Delta R\, g \,=\, \tfrac{1}{2} R^{2}\,g \,,
$$
or equivalently 
\begin{equation}\label{eq}
R \,Ric - \nabla^{2} R \,=\, \tfrac{3}{4(n-1)} R^{2}\,g \,,
\end{equation}
\begin{equation}\label{eq1}
\Delta R \,=\, \tfrac{n-4}{4(n-1)} R^{2} \,,
\end{equation}
where equation~\eqref{eq1} is just the trace of~\eqref{eq}. 

Obviously, if a metric is scalar flat or Einstein, then it satisfies~\eqref{eq}. Moreover, equation~\eqref{eq1} implies that any compact critical metric of $\mathcal{S}^{2}$ is {\em trivial}, in the sense that it has constant scalar curvature. More precisely, any compact critical metric of $\mathcal{S}^{2}$ is scalar flat if $n\neq 4$, whereas it is either scalar flat or Einstein if $n=4$. This is clear, since $\mathcal{S}^{2}$ is not scale--invariant if $n\neq4$. To obtain nontrivial compact critical metrics in this case, one should consider the modified functional
$$
\mathcal{S}_{v}^{2}(g) \,=\, \frac{\int_{M}R^{2}\,dV_{g}}{\big(\int_{M}dV_{g}\big)^{(n-4)/n}} \,,
$$
which is scale--invariant. A simple integration by parts argument (see the appendix) shows that any compact critical metric of $\mathcal{S}_{v}^{2}$ with nonnegative scalar curvature is either scalar flat or Einstein (see Anderson~\cite[Proposition 1.1]{and3} for a proof in dimension three). The existence of nontrivial compact critical metrics of $\mathcal{S}_{v}^{2}$ is still an open question.

In this paper we will focus on complete, possibly noncompact, critical metrics of $\mathcal{S}^{2}$. As far as we know, noncompact critical metrics of $\mathcal{S}^{2}$ were not studied yet. Our main result characterizes critical metrics with positive scalar curvature.

\begin{teo}\label{main} Let $(M^{n},g)$, $n\geq 3$, be a complete critical metric of $\mathcal{S}^{2}$ with positive scalar curvature. Then $(M^{n},g)$ has constant scalar curvature.
\end{teo}
In particular, from equations~\eqref{eq} and~\eqref{eq1}, if $n\neq 4$, there are no complete critical metrics of $\mathcal{S}^{2}$ with positive scalar curvature, whereas, every complete four--dimensional critical metric of $\mathcal{S}^{2}$ with positive scalar curvature is Einstein. Furthermore, by equation~\eqref{eq1} and the strong maximum principle, if $n\leq 4$ and $g$ is a critical metric with nonnegative scalar curvature, then either $R\equiv 0$ or $R>0$ on $M^{n}$. As an immediate consequence, we have the following characterization of complete critical metrics of $\mathcal{S}^{2}$ with nonnegative scalar curvature in dimension three and four. 

\begin{teo}\label{main3D} Let $(M^{3},g)$ be a complete three--dimensional  critical metric of $\mathcal{S}^{2}$ with nonnegative scalar curvature. Then $(M^{3},g)$ is scalar flat.  
\end{teo}

\begin{teo}\label{main4D} Let $(M^{4},g)$ be a complete four--dimensional  critical metric of $\mathcal{S}^{2}$ with nonnegative scalar curvature. Then $(M^{4},g)$ is either scalar flat or Einstein with positive scalar curvature.
\end{teo}

We do not know if the condition of nonnegative scalar curvature is necessary or can be dropped from these two theorems. Theorem~\ref{main3D} has to be compared with a result of Anderson in~\cite{and1} concerning the characterization of three--dimensional critical metric of the $L^{2}$--norm of the Ricci curvature 
$$
\mathcal{R}^2(g) \,=\, \int_{M} |Ric|^{2}\,dV_{g} \,.
$$
In fact, Anderson in~\cite[Theorem 0.1]{and1} proved that every complete three--dimensional critical metric of $\mathcal{R}^{2}$ with nonnegative scalar curvature is flat.

The proof of Theorem~\ref{main} relies on a gradient estimate for the scalar curvature of critical metrics and it is inspired by the classical Yau's estimate for positive harmonic functions on complete Riemannian manifolds with nonnegative Ricci curvature~\cite{yau2}.

\

\section{Proof of Theorem~\ref{main}}

Let $(M^{n},g)$ be a complete Riemannian manifold satisfying~\eqref{eq} 
$$
R \,Ric - \nabla^{2} R \,=\, \tfrac{3}{4(n-1)} R^{2}\,g \,.
$$
We recall the traced equation~\eqref{eq1}
$$
\Delta R \, = \, \tfrac{n-4}{4(n-1)} R^{2} \,.
$$
If $(M^{n},g)$ is compact, the maximum principle implies that the scalar curvature of $g$ has to be constant and Theorem~\ref{main} follows. From now on we will assume that $(M^{n},g)$ is a complete, noncompact, critical metric of $\mathcal{S}^{2}$ with positive scalar curvature, $R>0$.   

Let us define $u=\log R$. From equation~\eqref{eq1}, the function $u$ satisfies
$$
\Delta u \,=\, -|\nabla u|^{2}+\tfrac{n-4}{4(n-1)}R \,.
$$  
Moreover, using Bochner formula, we can compute
\begin{eqnarray*}
\Delta |\nabla u|^{2} &=& 2 |\nabla^{2} u|^{2} + 2 Ric(\nabla u, \nabla u) + 2 \langle \nabla u, \nabla \Delta u \rangle \\
&=& 2|\nabla^{2}u|^{2} + 2 Ric(\nabla u,\nabla u) - 2\langle \nabla u, \nabla |\nabla u|^{2} \rangle + \tfrac{n-4}{2(n-1)}R|\nabla u|^{2} \\
&=& 2|\nabla^{2}u|^{2}+ \tfrac{2}{R} \nabla^{2}R (\nabla u,\nabla u) + \tfrac{1}{2}R|\nabla u|^{2} - 2\langle \nabla u, \nabla |\nabla u|^{2} \rangle \,,
\end{eqnarray*}
where in the last equality we have used the structure equation~\eqref{eq}. On the other hand, one has
\begin{eqnarray*}
2\nabla^{2}R(\nabla u,\nabla u) &=& 2R^{-2}\nabla^{2}R(\nabla R,\nabla R)\\ 
&=& R^{-2}\langle \nabla R, \nabla |\nabla R|^{2} \rangle \\
&=& R\langle \nabla u ,\nabla |\nabla u|^{2} \rangle + 2 R |\nabla u|^{4} \,.
\end{eqnarray*}
Hence, by the previous computation, we have obtained
\begin{equation}\label{c1}
\Delta |\nabla u|^{2} \,=\, 2|\nabla^{2}u|^{2} - \langle \nabla u, \nabla |\nabla u|^{2} \rangle + \tfrac{1}{2}R|\nabla u|^{2} + 2|\nabla u|^{4} \,.
\end{equation}
Moreover, the standard matrix inequality $|A|^{2}\geq (1/n)\, tr(A)^{2}$, implies at once that
$$
2|\nabla^{2} u|^{2} \,\geq\, \tfrac{2}{n}|\nabla u|^{4} + \tfrac{(n-4)^{2}}{8n(n-1)^{2}}R^{2}-\tfrac{n-4}{n(n-1)}R|\nabla u|^{2} \,.
$$ 
Combining this estimate with~\eqref{c1} yields
\begin{equation}\label{c2}
\Delta |\nabla u|^{2} \,\geq\, - \langle \nabla u, \nabla |\nabla u|^{2} \rangle + \tfrac{n^{2}-3n+8}{2n(n-1)}R|\nabla u|^{2} + \tfrac{2(n+1)}{n}|\nabla u|^{4} + \tfrac{(n-4)^{2}}{8n(n-1)^{2}}R^{2}\,.
\end{equation}
Choose now $\f$ to be a nonnegative cut--off function on $M^{n}$ and let $H = \f|\nabla u|^{2}$. Then, at any point where $\f>0$, estimate~\eqref{c2} implies
\begin{eqnarray*}
\Delta H &=& (\Delta\f) \,|\nabla u|^{2} + \f \,\Delta |\nabla u|^{2} + 2 \langle \nabla \f,\nabla |\nabla u|^{2} \rangle \\
&=& (\Delta \f) \,\f^{-1} H + \f\, \Delta |\nabla u|^{2} + 2 \f^{-1} \langle \nabla \f,\nabla H \rangle - 2 |\nabla \f|^{2} \f^{-2} H \\
&\geq & (\Delta \f) \,\f^{-1} H  + 2 \f^{-1} \langle \nabla \f,\nabla H \rangle - 2 |\nabla \f|^{2} \f^{-2} H - \langle \nabla u ,\nabla H \rangle \\
&& +\, \f^{-1} H \langle \nabla u , \nabla \f \rangle + \tfrac{n^{2}-3n+8}{2n(n-1)}R\, H + \tfrac{2(n+1)}{n} \f^{-1} H^{2} + \tfrac{(n-4)^{2}}{8n(n-1)^{2}} \f \, R^{2} \,.
\end{eqnarray*}
Moreover, Cauchy--Schwartz inequality gives 
$$
\f^{-1} H \langle \nabla u,\nabla \f \rangle \,\geq \, - |\nabla \f|\, 	\f^{-3/2} \, H^{3/2} \,
$$
and we have
\begin{eqnarray*}
\Delta H &\geq & (\Delta \f) \,\f^{-1} H  + 2 \f^{-1} \langle \nabla \f,\nabla H \rangle - 2 |\nabla \f|^{2} \f^{-2} H - \langle \nabla u ,\nabla H \rangle \\
&& -\, |\nabla \f|\, 	\f^{-3/2} \, H^{3/2} + \tfrac{n^{2}-3n+8}{2n(n-1)}R\, H + \tfrac{2(n+1)}{n} \f^{-1} H^{2} + \tfrac{(n-4)^{2}}{8n(n-1)^{2}} \f \, R^{2} \,.
\end{eqnarray*}
Thus, at a maximum point $p_{0}\in M^{n}$ of $H$, one has
\begin{eqnarray}\label{c3} \nonumber
0 &\geq & (\Delta \f) H - 2 |\nabla \f|^{2} \f^{-1} H - |\nabla \f|\, 	\f^{-1/2} \, H^{3/2} \\\nonumber
&& + \,\tfrac{n^{2}-3n+8}{2n(n-1)}\f\,R\, H + \tfrac{2(n+1)}{n} H^{2} + \tfrac{(n-4)^{2}}{8n(n-1)^{2}} \f^{2} \, R^{2} \\
&\geq &  (\Delta \f) H - 2 |\nabla \f|^{2} \f^{-1} H - |\nabla \f|\, 	\f^{-1/2} \, H^{3/2} + \tfrac{2(n+1)}{n} H^{2} \,,
\end{eqnarray}
where we have used the fact that $R>0$. 

Let $\f=\f(r)$ be a function of the distance $r$ to a fixed point $p\in M^{n}$ and let $B_{s}(p)$ be a geodesic ball of radius $s$. We denote by $C_{p}$ the cut locus at the point $p$ and we choose $\f$ satisfying the following properties: $\f=1$ on $B_{s}(p)$, $\f=0$ on $M^{n} \setminus B_{2s}(p)$ and
$$
-c\,s^{-1} \f^{1/2} \,\leq\, \f' \leq 0 \quad \quad\hbox{and} \quad \quad |\f''| \,\leq\, c \,s^{-2} \,
$$
on $B_{2s}(p) \setminus B_{s}(p)$ for some positive constant $c>0$. In particular, $\f$ is smooth in $M^{n}\setminus C_{p}$ and in $\{B_{2s}(p) \setminus B_{s}(p)\}\setminus C_{p}$ one has
\begin{equation}\label{c4}
|\nabla \f|\,\f^{-1/2} \,\leq\, |\f'|\,\f^{-1/2} \,\leq\, c\, s^{-1} \,.
\end{equation}
Hence, to conclude the proof it remains to estimate the Laplacian term $\Delta \f$. Notice that
$$
\Delta \f \,=\, \f' \Delta r + \f'' \,.
$$
Let $v=-u=-\log R$. One has
$$
\nabla^{2}v - dv \otimes dv \,=\, - R^{-1} \nabla^{2} R \,.
$$
Thus, by the structure equation~\eqref{eq}, we obtain that the metric $g$ satisfies 
$$
Ric + \nabla^{2}v - dv\otimes dv \,=\, \tfrac{3}{4(n-1)} e^{-v} \,g \geq 0 
$$
In particular, following the notations in~\cite{weiwyl}, the $1$--Bakry--Emery Ricci tensor $Ric^{1}_{v}$ of $g$ defined by
$$
Ric^{1}_{v} \,=\, Ric + \nabla^{2}v - dv\otimes dv 
$$
is nonnegative. Hence, by the Laplacian comparison estimate on manifolds with nonnegative $1$--Bakry--Emery Ricci tensor~\cite[Theorem A.1]{weiwyl}, for every $x\in\{B_{2s}(p) \setminus B_{s}(p)\}\setminus C_{p}$, one has
\begin{eqnarray*}
\Delta r &\leq& \langle \nabla r, \nabla v \rangle + n\,r^{-1} \\
&\leq& |\nabla u| + n\,r^{-1} \\
&=& \f^{-1/2}\,H^{1/2} + n\,s^{-1} \,,
\end{eqnarray*}
since $s\leq r$. In particular, for every $x\in\{B_{2s}(p) \setminus B_{s}(p)\}\setminus C_{p}$, we obtain
\begin{eqnarray*}
\Delta \f &=& \f' \Delta r + \f'' \\
&\geq& \f'\,\f^{-1/2} H^{1/2}+ n\f' \, s^{-1} - c\, s^{-2} \\
&\geq& -c\,s^{-1} H^{1/2}- nc\,s^{-2} - c\,s^{-2} \\
&=& -c\,s^{-1} H^{1/2} - C_{1}\, s^{-2} \,, 
\end{eqnarray*}
for some positive constant $C_{1}>0$. Let us assume that the maximum point $p_{0}$ of $H$ does not belong to the cut locus $C_{p}$ of $p$. Combining the last estimate with~\eqref{c3} and~\eqref{c4}, at $p_{0}\in M^{n}$, we get
$$
0 \,\geq\, -C_{2}\,s^{-2} H - C_{3}\,s^{-1} H^{3/2} + \tfrac{2(n+1)}{n} H^{2} \,,
$$
for some positive constants $C_{2},C_{3}>0$. On the other hand, 
$$
C_{3}\,s^{-1} H^{3/2} \,\leq\,  \alpha H^{2} + \tfrac{C_{3}^{2}}{4} \alpha^{-1}\,s^{-2}\, H \,
$$
for every $\alpha>0$. Thus, if $\alpha< 2(n+1)/n$, we have proved that, if $p_{0}\notin C_{p}$, then
$$
H(p_{0}) \,\leq\, C_{4} \, s^{-2} \,,
$$
for some positive constants $C_{4}>0$. By letting $s \rightarrow +\infty$ we obtain that $H\equiv 0$, so $u$ is constant on $M^{n}$ and $g$ has constant scalar curvature. 

If $p_{0}\in C_{p}$ we argue as follows (this trick is usually referred to Calabi). Let $\gamma:[0,L] \rightarrow M^{n}$, where $L=d(p_{0},p)$, be a minimal geodesic joining $p$ to $p_{0}$, the maximum point of $H$. Let $p_{\eps}=\gamma(\eps)$ for some $\eps>0$. Define now
$$
H_{\eps} \,=\, \f \big(d(x,p_{\eps})+\eps\big) |\nabla u|^{2}\,.
$$
Since $d(x,p_{\eps})+\eps\geq d(x,p)$ and $d(p_{0},p_{\eps}) + \eps = d(p_{0},p)$, it is easy to see that $H_{\eps}(p_{0}) = H (p_{0})$ and 
$$
H_{\eps}(x) \,\leq \, H(x) \quad\quad\hbox{for all }\,x \in M^{n} \,, 
$$
since $\f'\leq 0$. Hence $p_{0}$ is also a maximum point for $H_{\eps}$. Moreover, since $d(x,p_{\eps})$ is smooth in a neighborhood of $p_{0}$ we can apply the maximum principle argument as before to obtain an estimate for $H_{\eps}(p_{0})$ which depends on $\eps$. Taking the limit as $\eps\rightarrow 0$, we obtain the desired estimate on $H$.

This concludes the proof of Theorem~\ref{main}. As we have observed in the introduction, Theorem~\ref{main3D} and Theorem~\ref{main4D} now follows simply by observing that, if $n\leq 4$ and $g$ is a critical metric with nonnegative scalar curvature, then the strong maximum principle implies that either $R\equiv 0$ or $R>0$ on $M^{n}$.

\

\section{Appendix}

Let $(M^{n},g)$, $n\geq 3$, be a compact Riemannian manifold and consider the scale--invariant functional
$$
\mathcal{S}_{v}^{2}(g) \,=\, \frac{\int_{M}R^{2}\,dV_{g}}{\big(\int_{M}dV_{g}\big)^{(n-4)/n}} \,.
$$
The Euler--Lagrange equation for a critical metric of $\mathcal{S}_{v}^{2}$ is given by
\begin{equation}\label{eqc}
R \,Ric - \nabla^{2} R \,=\, \tfrac{1}{n} R^{2}\,g -\tfrac{1}{n} \Delta R \,g\,.
\end{equation}
Taking the divergence of~\eqref{eqc}, one has
\begin{eqnarray*}
0 &=& Ric(\nabla R,\cdot) + \tfrac{1}{2} R\,\nabla R - \Delta \nabla R -\tfrac{2}{n} R\,\nabla R + \tfrac{1}{n}\nabla \Delta R \\
&=& Ric(\nabla R, \cdot) +\tfrac{n-4}{4n} \nabla R^{2} - \tfrac{n-1}{n} \nabla \Delta R - Ric(\nabla R, \cdot) \\
&=& \tfrac{1}{4n} \nabla \big( (n-4) R^{2} - 4(n-1)\Delta R \big) \,,
\end{eqnarray*}
where we have used Schur's identity $dR= 2 \,\hbox{div} Ric$ and the commutation formula for covariant derivatives. Hence, one has that any solution of~\eqref{eqc}, satisfies
\begin{equation}\label{eqc1}
\Delta R \,=\, \tfrac{n-4}{4(n-1)} \big( R^{2} - \overline{R^{2}} \big)\,,
\end{equation}
where $\overline{R^{2}}=\big(\int_{M} R^{2}\,dV_{g}\big) / \big(\int_{M} dV_{g}\big)$. 
Obviously, if a metric is scalar flat or Einstein, then it  critical for $\mathcal{S}^{2}_{v}$. We prove that also the converse is true, if we assume that the critical metric has nonnegative scalar curvature. We notice that equation~\eqref{eqc1} implies that any four dimensional critical metric of $\mathcal{S}^{2}_{v}$ has constant scalar curvature and it is either scalar flat or Einstein. 
\begin{prop} Any compact critical metric of $\mathcal{S}^{2}_{v}$ with nonnegative scalar curvature either is scalar flat or Einstein.
\end{prop}
\begin{proof} Contracting equation~\eqref{eqc} with the Ricci tensor, one has
$$
R \,\big|Ric - \tfrac{1}{n}R\,g\big|^{2} \,=\, R^{ij}\nabla_{ij} R -\tfrac{1}{n} R \,\Delta R \,.  
$$
Integrating on $M^{n}$, we get
\begin{eqnarray*}
\int_{M} R \,\big|Ric - \tfrac{1}{n}R\,g\big|^{2} \,dV_{g} &=& \int_{M} R^{ij} \nabla_{ij}R \,dV_{g} - \tfrac{1}{n} \int_{M} R \Delta R \,dV_{g} \\
&=& -\tfrac{1}{2} \int_{M} |\nabla R|^{2} +\tfrac{1}{n} \int_{M} |\nabla R|^{2} \,dV_{g} \\
&=& - \tfrac{n-2}{2n} \int_{M} |\nabla R|^{2} \,dV_{g} \,.
\end{eqnarray*}
Since $R\geq 0$, this implies that either $R\equiv 0$ or the metric $g$ is Einstein.
\end{proof}

As we have observed in the introduction, this result was proved in dimension three by Anderson in~\cite{and3}.

\

\

\begin{ackn}
\noindent The author is partially supported by the Italian project FIRB--IDEAS ``Analysis and Beyond''.
\end{ackn}

\bigskip

\bigskip

\bibliographystyle{amsplain}
\bibliography{critical}

\bigskip

\bigskip

\parindent=0pt

\end{document}